\documentclass[10pt,twoside]{siamltex}
\usepackage{amsfonts}
\usepackage{mathrsfs}
\usepackage{amsfonts,epsfig}

\newtheorem{remark}[theorem]{Remark}
\newtheorem{algorithm}[theorem]{Algorithm}

\newcommand{\complex}{\mathbb{C}}

\begin{document}



\bibliographystyle{plain}
\title{On parallel multisplitting methods for non-Hermitian positive definite\\ linear systems}

\author{
Cheng-yi Zhang\thanks{Institute of Information and system Science,
Xi'an Jiaotong University, Xi'an, Shaanxi, 710049, P.R. China;
School of Science, Xi'an Polytechnic University, Xi'an, Shaanxi,
710048, P.R. China (chyzhang08@126.com). Supported by the Science
Foundation of the Education Department of Shaanxi Province of China\
(11JK0492), the Scientific Research Foundation of Xi'an Polytechnic
University\ (BS1014) and China Postdoctoral Science
Foundation(20110491668).}
\and Shuanghua Luo\thanks {Department of Mathematics of School of
Science, Xi'an Polytechnic University, Xi'an, Shaanxi 710048, P.R.
China (iwantflyluo@163.com).}\and Yan Zhu\thanks{College of
Mathematic and Information Science, Qujing Normal University,
Qujing, 65011, P.R. China (zhuyanlj@163.com).}
}

\pagestyle{myheadings} \markboth{Cheng-yi Zhang, Shuanghua Luo and
Yan Zhu}{On paralell multisplitting methods for non-Hermitian
positive definite linear systems} \maketitle

\begin{center}
\end{center}

\begin{abstract}
To solve non-Hermitian linear system $Ax=b$ on parallel and vector
machines, some paralell multisplitting methods are considered. In
this work, in particular: i) We establish the convergence results of
the paralell multisplitting methods, together with its relaxed
version, some of which can be regarded as generalizations of
analogous results for the Hermitian positive definite case; ii) We
extend the positive-definite and skew-Hermitian splitting (PSS)
method methods in [{\em SIAM J.~Sci.~Comput.}, 26:844--863, 2005] to
the parallel PSS methods and propose the
corresponding convergence results.
\end{abstract}

\section{Introduction} \label{intro-sec}
Many problems in scientific computing give rise to a system of $n$
linear equations in $n$ unknowns,
\begin{equation}\label{r1}
Ax=b,\ \ \ A=(a_{ij})\in \complex^{n\times n}\ \ {\rm nonsingular,\
\ and} \ b,x\in \complex^{n},
\end{equation} where $A$ is a large, sparse non-Hermitian matrix.
In this paper we consider the important case where $A$ is {\em
positive definite}; i.e., the Hermitian part $H=(A+A^*)/2$ is
Hermitian positive definite, where $A^*$ denotes the conjugate
transpose of the matrix $A$. Large, sparse systems of this type
arise in many applications, including discretizations of
convection-diffusion problems \cite{elman}, regularized weighted
least-squares problems \cite{benzi_ng}, real-valued formulations of
certain complex symmetric systems \cite{bb}, and so forth.

In order to solve system (\ref{r1}) iteratively on parallel and
vector machines, O'Leary and White \cite{{o'leary}} introduced the
multisplitting technique for linear system. Later, this technique
was further studied by many authors; see e.g. \cite{{zhongzhibai}},
\cite{{frommer1}}, \cite{{frommer2}}, \cite{{neumann}},
\cite{{white1}}, \cite{{white2}},\cite{{white3}},
\cite{{hadjidimos}}, \cite{{guangxicao}}, \cite{{zhihaocao}},
\cite{{nabben}}, \cite{{elsner}}, \cite{{yongzhongsong1}},
\cite{{yongzhongsong2}}, \cite{{yongzhongsong3}}, \cite{{wenli}},
\cite{{daweichang}}, \cite{{jaeheonyun}}, \cite{{jaeheonyun1}},
\cite{jaeheonyun2} \cite{{wangderen}}, \cite{{wangxinmin}},
\cite{{joanjosepcliment}}, etc.

As defined in \cite{o'leary} and \cite{yongzhongsong2} a
multisplitting of $A$ is a collection of triples of matrices
$(M_k,N_k,E_k)_{k=1}^m$ satisfying\begin{itemize}
\item The matrix $A$ can be split into \begin{equation}\label{multi0}
A=M_k-N_k,\ k=1,2,\cdots,m,
\end{equation} where $M_k$ is nonsingular;
\item $E_k,\ k=1,2,\cdots,m$, are diagonal matrices with nonnegative
entries and satisfy $\sum_{k=1}^mE_k=I$, the identity matrix.
\end{itemize}

\begin{algorithm} \label{algorithm1}
Given any initial vector $x^{(0)}$.
\begin{itemize}
\item For $i=0,1,2,\cdots,$ until convergent.
\item For $k=1$ to $m$
\begin{equation}\label{multi01}
M_ky_k=N_kx^{(i)}+b ~~~~~~~~~~~
\end{equation}
\begin{equation}\label{multi02}
x^{(i+1)}=\sum_{k=1}^mE_ky_k. ~~~~~~~~~~~~~~~~~~~~~~
\end{equation}
\end{itemize}
 \end{algorithm}

It is easy to see that Algorithm \ref{algorithm1} corresponds to the
following iteration
\begin{equation}\label{multi1}
x^{(i+1)}=\sum_{k=1}^mE_kP_kx^{(i)},\ \ \ \ i=0,1,2,\ldots
\end{equation}
where the operators $P_k:\ \complex^n\rightarrow \complex^n,\ 1\leq
k\leq m,$ are defined as
\begin{equation}\label{multi2}
P_kx=M_k^{-1}N_kx+M_k^{-1}b.
\end{equation}

Thus, iteration (\ref{multi1}) can be rewritten as
\begin{equation}\label{multi3}
x^{(i+1)}=Tx^{(i)}+\sum_{k=1}^mE_kM_k^{-1}b,\ \ \ \ i=0,1,2,\ldots
\end{equation}
where $T=\sum_{k=1}^mE_kM_k^{-1}N_k$ is the iteration matrix.

Conditions on the splittings (\ref{multi0}) and on the weighting
matrices which ensure the convergence of Algorithm \ref{algorithm1}
in some important cases where given by O'Leary and White
\cite{{o'leary}}, Nabben \cite{{nabben}}, Neumann and Plemmons
\cite{{neumann}}, Frommer et al \cite{frommer1}, \cite{frommer2},
Song et al \cite{{yongzhongsong1}}, \cite{{yongzhongsong2}},
\cite{{yongzhongsong3}}, Li et el \cite{{wenli}}, Hadjidimos and
Yeyios \cite{hadjidimos}, Cao and Song \cite{guangxicao}, etc. They
showed that Algorithm \ref{algorithm1} (semi)converges when
\begin{itemize}
\item $A$ is Hermitian (or symmetric) positive definite and the
splittings (\ref{multi0}) are $P-$regular;
\item $A$ is monotone and the splittings (\ref{multi0}) are (weak) regular;
\item $A$ is an $H-$matrix and the splittings (\ref{multi0}) are $H-$compatible splittings \cite{jaeheonyun2};
\item $A$ is Hermitian (or symmetric) positive semidefinite and the
splittings (\ref{multi0}) are $P-$regular;
\item $A$ is a singular $M-$matrix and the splittings (\ref{multi0}) are (weak)
regular.
\end{itemize}

Recently, there has been considerable interest in the
positive-definite and skew-Hermitian splitting (PSS) method
introduced by Bai, Golub, Lu and Yin for solving non-Hermitian
positive definite linear systems, see \cite{B.Z3}. In this paper we
further study this method and generalize it to the parallel PSS
method. Let \begin{equation}\label{multi05}
\begin{array}{lll}
A=M_k+N_k&=&(M_k+N_k^*)+(N_k-N_k^*)\\
&=&P_k+S_k,\\
&&\ \ \ \ \ k=1,2,\cdots,m
\end{array}
\end{equation}
where $M_k\neq 0$ and $N_k\neq 0$. If $A$ is non-Hermitian positive
definite, so is $P_k:=M_k+N_k^*$. Furthermore, $S_k:=N_k-N_k^*$ is
skew-Hermitian. Thus, $A=P_k+S_k,\ k=1,2,\cdots,m,$ are PS
splittings.

\begin{algorithm} \label{algorithm2} {\rm(}Parallel PSS method{\rm)}
Given any initial vector $x^{(0)}$.
\begin{itemize}
\item For $i=0,1,2,\cdots,$ until convergent.
\item For $k=1$ to $m$
\begin{equation}\label{multi5}
\left\{
\begin{array}{cc}
(\alpha_k I+P_k)x^{(i+1/2)}=(\alpha_k I-S_k)x^{(i)}+b\\
(\alpha_k I+S_k)y_k=(\alpha_k I-P_k)x^{(i+1/2)}+b~
\end{array}
\right.
\end{equation}
\begin{equation}\label{multi6}
x^{(i+1)}=\sum_{k=1}^mE_ky_k. ~~~~~~~~~~~~~~~~~~~~~~
\end{equation}
\end{itemize}
 \end{algorithm}

In matrix-vector form and for each $k,\ k=1,2,\cdots,m$, the PSS
iteration (\ref{multi5}) can be equivalently rewritten as
\begin{equation}\label{multi7}
y_k=M(\alpha_k)x^{(i)}+G(\alpha_k)b,
\end{equation}
where
\begin{equation}\label{multi8}
\left\{
\begin{array}{ll}
M(\alpha_k)=(\alpha_k I+S_k)^{-1}(\alpha_k I-P_k)(\alpha_k I+P_k)^{-1}(\alpha_k I-S_k)\\
G(\alpha_k)=2\alpha_k (\alpha_k I+S_k)^{-1}(\alpha_k I+P_k)^{-1}
\end{array}
\right.
\end{equation}

Thus, Algorithm \ref{algorithm2} can be rewritten as the following
iteration scheme
\begin{equation}\label{multi9}
x^{(i+1)}=\mathscr{M}(\alpha)x^{(i)}+\sum_{k=1}^mE_kG(\alpha_k)b,\ \
\ \ i=0,1,2,\ldots
\end{equation}
where
\begin{equation}\label{multi10}
\begin{array}{lll}
\mathscr{M}(\alpha)&=&\sum_{k=1}^mE_kM(\alpha_k)\\&=&\sum_{k=1}^mE_k(\alpha_k
I+S_k)^{-1}(\alpha_k I-P_k)(\alpha_k I+P_k)^{-1}(\alpha_k I-S_k)
\end{array}
\end{equation}
 is the iteration matrix.

When the matrix $N_k$ in (\ref{multi05}) is triangular or block
triangular, the splittings (\ref{multi05}) are TS splittings or BTS
splittings, and thus, Algorithm \ref{algorithm2} becomes Parallel
TSS method or Parallel BTSS method.

In Algorithms \ref{algorithm1} and \ref{algorithm2} a relaxation
parameter $\omega\in \mathbb{R},\ \omega\neq0,$ can be introduced by
replacing the computation of $x^{(i+1)}$ in (\ref{multi02}) with the
equation
\begin{equation}\label{multi4}
\begin{array}{lll}
x^{(i+1)}&=&\omega \sum_{k=1}^mE_ky_k+(1-\omega)x^{(i)}\\
&=&T_{\omega}x^{(i)}+\omega\sum_{k=1}^mE_kM_k^{-1}b,\ \ \ \
i=0,1,2,\ldots
\end{array}
\end{equation}
where $T_{\omega}=\omega X+(1-\omega)I$ is the iteration matrix with
either $X=T$ or $X=\mathscr{M}(\alpha)$. Clearly, with $\omega=1$,
equation (\ref{multi02}) is recovered. In the case of $\omega\neq1$,
we have a Relaxed Multisplitting (see \cite{frommer1} and
\cite{o'leary}) or a Relaxed Parallel PSS (TS, BTSS) Algorithm.

There have been several studies on the convergence of multisplitting
iterative methods for non-Hermitian positive definite linear
systems. In \cite{{hadjidimos}} and \cite{{wangchuanlong}} some
convergence conditions of multisplitting methods for non-Hermitian
positive definite matrices have been established.

Continuing in this direction, in this paper we establish new results
on multisplitting methods for solving system (\ref{r1}) iteratively,
focusing on a particular class of splittings. For a given matrix
$A\in \complex^{n\times n}$, a splitting $A=M-N$ with $M$
nonsingular is called a {\em $P$-regular splitting} if the matrix
$M^*+N$ is positive definite, i.e., the Hermitian part of $M^*+N$ is
Hermitian positive definite \cite{ortega}. It is a well known result
\cite{Wei,ortega} that if $A$ is Hermitian positive definite and
$A=M-N$ is a $P$-regular splitting, then the splitting iterative
method is convergent: $\rho(M^{-1}N)<1$. An extension of {\em
$P$-regular splitting} was introduced by Ortega and Plemmons
\cite{ortega11} and \cite{A.B}. A splitting $A=M-N$ with $M$
nonsingular is called an {\em extended $P$-regular splitting} if the
matrix $M^*(A^{-1})^*A+N$ is positive definite. A stronger condition
of the splitting $A=M-N$ proposed by Yuan \cite{jinyunyuan} that
$M^*A+A^*N$ is positive definite guarantees that the splitting
iterative method is convergent. In this paper, we propose some
conditions such that the parallel multisplitting methods converge by
examining the spectral properties of the iteration matrix induced by
these special multisplittings of a non-Hermitian positive definite
matrix.

The paper is organized as follows. Some notations and preliminary
results are given in Section 2. In section 3 we study the
convergence of Algorithm \ref{algorithm1}, together with its relaxed
version. In section 4 we discuss the convergence of Algorithm
\ref{algorithm2}. Some conclusions are given in section 5.

\section{Notation and preliminaries}\label{prelimi-sec}
For convenience, some of the terminology used in this paper will be
given.

The symbol $\complex^{n\times n}$ will denote the set of all
$n\times n$ complex matrices. Let $A,\ B\in \complex^{n\times n}$.
We use the notation $A\succ 0$ ($A\succeq 0$) if $A$ is Hermitian
positive (semi-)definite. If $A$ and $B$ are both Hermitian, we
write $A\succ B$ ($A\succeq B$) if and only if $A-B\succ 0$
($A-B\succeq 0$). If $A$ is Hermitian matrix, then all of
eigenvalues of $A$ are real, and we denote by $\lambda_{\min}(A)$
and $\lambda_{\max}(A)$ the smallest (i.e., leftmost) and largest
(rightmost) eigenvalues, respectively. Let $A\in \complex^{n\times
n}$ with $H=(A+A^*)/2$ and $S=(A-A^*)/2$ its Hermitian and
skew-Hermitian parts, respectively; then $A$ is non-Hermitian
positive (semi-)definite if and only if $H\succ0$ ($H\succeq0$).
Furthermore, $\|A\|_2=\sqrt{\lambda_{\max}(A^*A)}$, denotes the
spectral norm of the matrix $A$.

The following theorems gives convergence conditions for iterative
methods based on a single splitting $A=M-N$.

\begin{theorem}\label{main_thm} {\rm (see \cite{chengyi})}
Let $A\in \complex^{n\times n}$ be non-Hermitian positive definite,
and let $A=M-N$ be a $P$-regular splitting with $N$ Hermitian. Then
$\rho(M^{-1}N)<1.$
\end{theorem}

The proof can be found, e.g, in \cite{chengyi}.

\begin{corollary} {\rm (see \cite{chengyi})}
Let $A\in \complex^{n\times n}$ be non-Hermitian positive definite,
and let $A=M-N$ be a splitting with $N\succeq0$. Then
$\rho(M^{-1}N)<1.$
\end{corollary}

\begin{theorem} \label{ortega} {\rm (see \cite{{{ortega11}}} and \cite{A.B})}
Let $A\in \complex^{n\times n}$ such that $A=M-N$ is an extended
P-regular splitting. Then $\rho(T)<1,$ where $T=M^{-1}N$, if and
only if $A$ is positive definite.
\end{theorem}

\begin{remark}
The condition that $A=M-N$ is an extended P-regular splitting can be
replaced by the condition that $M+N^*(A^{-1})^*A$ is positive
definite since $A=M-N$ is an extended P-regular splitting,
\begin{equation}\label{new1}
\begin{array}{lll}
M^*(A^{-1})^*A+N&=&(A+N)^*(A^{-1})^*A+N\\
&=&M+N^*(A^{-1})^*A
\end{array}
\end{equation}  is positive definite.
\end{remark}

\begin{theorem} \label{theorem6} {\rm (see \cite{{{jinyunyuan}}})}
Let $A\in \complex^{n\times n}$ be nonsingular, and let $A=M-N$ such
that $M^*A+A^*N=M^*M-N^*N$ is positive definite. Then Then
$\rho(T)<1,$ where $T=M^{-1}N$.
\end{theorem}

A convergence result on multisplitting method for nonsymmetric
positive definite linear system is introduced by Hadjidimos and
Yeyios \cite{hadjidimos}.

\begin{theorem} \label{theorem7} {\rm (see \cite{hadjidimos})}
Let $A=(a_{ij})\in \mathbb{R}^{n\times n}$ be nonsymmetric positive
definite with a multisplitting $(M_k,N_k,E_k)_{k=1}^{m}$ satisfying
\begin{equation}\label{multi12}
\begin{array}{lll}
M_k&=&D+\rho_kI-L,\ \ N_k=\rho_kI+U,\ \ 1\leq k\leq m;\\
M_k&=&D+\rho_kI-U,\ \ N_k=\rho_kI+L,\ \ m+1\leq k\leq 2m;\\
\rho_k&>&\left\{
\begin{array}{ll}
max\{0,-\eta_m/\lambda_m\} for 1\leq k\leq m\\
max\{0,-\theta_m/\lambda_m\} for m+1\leq k\leq 2m
\end{array}
\right.,\ \ E_k=\alpha_kI,
\end{array}
\end{equation}
where $D=diag(A)$, $L,\ U$ are strictly lower and upper triangular
matrices satisfying $A=D-L-U$; $\lambda_m$ is the minimal eigenvalue
of $A+A^T$ and $\eta_m,\ \theta_m$ are the minimal eigenvalues of
the matrices $(D-L)(D-L)^T-UU^T$ and $(D-U)(D-U)^T-LL^T$,
respectively. Then Algorithm \ref{algorithm1} converges.

\end{theorem}

\section{Convergence of stationary multisplitting method}\label{convergence-sec}
In this section we discuss convergence of the parallel
mulitisplitting iterative methods for non-Hermitian linear systems,
especially, non-Hermitian positive definite linear systems.

\begin{theorem} \label{theorem8}
Let $A\in \complex^{n\times n}$ be nonsingular with a multisplitting
$(M_k,N_k,E_k)_{k=1}^{m}$ satisfying one of the following
conditions:

{\rm(i)} $E_k=\beta_k I$ and $M_k^*A+A^*N_k=M_k^*M_k-N_k^*N_k\succ0$
for $k=1,2,\cdots,m$;

{\rm(ii)} $E_k=\beta_k I$ and $\|M_k^{-1}N_k\|_2<1$ for
$k=1,2,\cdots,m$.\\
Then Algorithm \ref{algorithm1} converges to the unique solution of
{\rm(\ref{r1})} for any choice of the initial guess $x^{(0)}$.
\end{theorem}

\begin{proof} The iteration matrix of
Algorithm \ref{algorithm1} is $T=\sum_{k=1}^mE_kM_k^{-1}N_k$.
Algorithm \ref{algorithm1} is convergent by showing $\rho(T)<1.$

1) Assume that the condition (i) holds. Then, similar to the proof
of Theorem 1 in \cite{o'leary}, one has
\begin{equation}\label{love1}
\begin{array}{lll}
A^*A&-&T^*A^*AT\\ &=&A^*A-(\sum_{k=1}^mE_kM_k^{-1}N_k)^*A^*A(\sum_{k=1}^mE_kM_k^{-1}N_k)\\
&=&A^*A-(I-\sum_{k=1}^mE_kM_k^{-1}A)^*A^*A(I-\sum_{k=1}^mE_kM_k^{-1}A)\\
&=&\sum_{k=1}^m\beta_k(A^*AM_k^{-1}A+A^*(M_k^{-1})^*A^*A)\\&&-\sum_{k,j=1}^m\beta_k\beta_jA^*(M_k^{-1})^*A^*AM_j^{-1}A\\
&=&\sum_{k=1}^m\beta_kA^*(M_k^{-1})^*(M_k^{*}A+A^*M_k-\beta_kA^*A)M_k^{-1}A\\&&-\sum_{k,j=1;k\neq j}^m\beta_k\beta_jA^*(M_k^{-1})^*A^*AM_j^{-1}A\\
&=&\sum_{k=1}^m\beta_kA^*(M_k^{-1})^*(M_k^{*}A+A^*N_k+\sum_{j=1;j\neq k}^m\beta_jA^*A)M_k^{-1}A\\&&-\sum_{k,j=1;k\neq j}^m\beta_k\beta_jA^*(M_k^{-1})^*A^*AM_j^{-1}A\\
&=&\sum_{k=1}^m\beta_kA^*(M_k^{-1})^*(M_k^{*}A+A^*N_k)M_k^{-1}A\\&&+\sum_{k,j=1;k\neq j}^m\beta_k\beta_jA^*(M_k^{-1})^*A^*A(M_k^{-1}-M_j^{-1})A\\
&=&S_1+S_2,
\end{array}
\end{equation}
where
\begin{equation}\label{love2}
\begin{array}{lll}
S_1&=&\sum_{k=1}^m\beta_kA^*(M_k^{-1})^*(M_k^{*}A+A^*N_k)M_k^{-1}A\ \ \  \  {\rm and}\\
S_2&=&\sum_{k,j=1;k\neq
j}^m\beta_k\beta_jA^*(M_k^{-1})^*A^*A(M_k^{-1}-M_j^{-1})A.
\end{array}
\end{equation}
Since $A$ and $M_k$ are nonsingular,
$M_k^*A+A^*N_k=M_k^*M_k-N_k^*N_k\succ0$ for $k=1,2,\cdots,m$, it is
easy to see
\begin{equation}\label{love2b}S_1=\sum_{k=1}^m\beta_kA^*(M_k^{-1})^*(M_k^{*}A+A^*N_k)M_k^{-1}A\succ0.
\end{equation}
It is observed that
\begin{equation}\label{love2a}
\begin{array}{lll}
2S&=&S_2^*+S_2\\&=&\sum_{k,j=1;k\neq
j}^m\beta_k\beta_j[A^*(M_k^{-1}-M_j^{-1})^*A^*AM_k^{-1}A\\&&+A^*(M_k^{-1})^*A^*A(M_k^{-1}-M_j^{-1})A]\\
&=&\sum_{k,j=1;k\neq
j}^m\beta_k\beta_j[A^*(M_k^{-1}-M_j^{-1})^*A^*A(M_k^{-1}-M_j^{-1})A\\&&+A^*(M_k^{-1})^*A^*A(M_k^{-1}-M_j^{-1})A+A^*(M_k^{-1}-M_j^{-1})^*A^*AM_j^{-1}A]\\
&=&\sum_{k,j=1;k\neq
j}^m\beta_k\beta_jA^*(M_k^{-1}-M_j^{-1})^*A^*A(M_k^{-1}-M_j^{-1})A\\&&+\sum_{k,j=1;k\neq
j}^m\beta_k\beta_j[A^*(M_k^{-1})^*A^*AM_k^{-1}A-A^*(M_j^{-1})^*A^*AM_j^{-1}A]\\
&=&\sum_{k,j=1;k\neq
j}^m\beta_k\beta_jA^*(M_k^{-1}-M_j^{-1})^*A^*A(M_k^{-1}-M_j^{-1})A\\
&\succeq &0.
\end{array}
\end{equation}
As a result, (\ref{love1}), (\ref{love2b}) and (\ref{love2a})
indicate $A^*A-T^*A^*AT\succ 0.$ Since $A$ is nonsingular,
$A^*A\succ 0.$ It follows from Stein's Theorem (see, e.g.,
\cite{{ortega},{stein}}) that $T$ is convergent, i.e., $\rho(T)<1.$
Thus, Algorithm \ref{algorithm1} is convergent.

2) Now, assume that the condition (ii) holds, that is, $E_k=\beta_k
I$ and $\|M_k^{-1}N_k\|_2<1$ for $k=1,2,\cdots,m$. Thus, we have
\begin{equation}\label{love2c}
\begin{array}{lll}
\rho(T)&\leq &\|T\|_2=\|\sum_{k=1}^mE_kM_k^{-1}N_k\|_2\\
&\leq&\sum_{k,=1}^m\|E_kM_k^{-1}N_k\|_2\\
&=&\sum_{k=1}^m\beta_k\|M_k^{-1}N_k\|_2\\
&<&\sum_{k=1}^m\beta_k=1,
\end{array}
\end{equation}
which shows that Algorithm \ref{algorithm1} is convergent. This
completes the proof.
\end{proof}

\begin{remark}
If the multisplitting $(M_k,N_k,E_k)_{k=1}^{m}$ is defined by
(\ref{multi12}), we have $\|M_k^{-1}N_k\|_2<1$ for $k=1,2,\cdots,m$
(see the proof of Theorem 2.4 in \cite{hadjidimos}) and can obtain
the proof of Theorem \ref{theorem7} coming from Theorem
\ref{theorem8}. Therefore, similar to (\ref{multi12}), we can
construct a multisplitting $(M_k,N_k,E_k)_{k=1}^{m}$ to satisfy the
condition (ii) of Theorem \ref{theorem8}.
\end{remark}

In what follows some convergence results on parallel mulitisplitting
method for non-Hermitian positive definite linear systems will be
established. At first, the following lemma will be used in this
section.

\begin{lemma}\label{loveluo1} {\rm (see \cite{{A.B}})}
Let $A=M-N\in \complex^{n\times n}$ with $A$ and $M$ nonsingular and
let $T=M^{-1}N$. Then $A-T^*AT=(I-T^*)(M^*(A^{-1})^*A+N)(I-T)$.
\end{lemma}

\begin{theorem} \label{theorem9}
Let $A\in \complex^{n\times n}$ be non-Hermitian positive definite
with a multisplitting $(M_k,N_k,E_k)_{k=1}^{m}$ satisfying the
condition that $E_k=\beta_k I$ and $A=M_k-N_k$ is an extended
P-regular splitting for all $k=1,2,\cdots,m.$ Then Algorithm
\ref{algorithm1} converges to the unique solution of {\rm(\ref{r1})}
for any choice of the initial guess $x^{(0)}$.
\end{theorem}

\begin{proof}
Since the multisplitting $(M_k,N_k,E_k)_{k=1}^{m}$ satisfies the
condition that $E_k=\beta_k I$ and $A=M_k-N_k$ is an extended
$P-$regular splitting for all $k=1,2,\cdots,m,$
$M_k^*(A^{-1})^*A+N_k$ is positive definite. Let $T_k=M_k^{-1}N_k$
for all $k=1,2,\cdots,m$, then the iteration matrix of Algorithm
\ref{algorithm1} is
\begin{equation}\label{love3}T=\sum_{k=1}^mE_kT_k=\sum_{k=1}^m\beta_kT_k.
\end{equation}
It follows from Lemma \ref{loveluo1} that
\begin{equation}\label{love3b}
\begin{array}{lll}
A-T_k^*AT_k=(I-T_k^*)(M_k^*(A^{-1})^*A+N_k)(I-T_k).
\end{array}
\end{equation}
Again, Theorem \ref{ortega} shows that $\rho(T_k)<1$, and
consequently $I-T_k$ is nonsingular. Since $M_k^*(A^{-1})^*A+N_k$ is
positive definite, (\ref{love3b}) shows that $A-T_k^*AT_k$ is also
positive definite for $k=1,2,\cdots,m$. Noting that $A$ is positive
definite, its Hermitian part $H=(A+A^*)/2\succ0.$ Then,
(\ref{love3b}) shows \begin{equation}\label{love3c}
\begin{array}{lll}
H-T_k^*HT_k\succ0,\ \ \ \ k=1,2,\cdots,m,
\end{array}
\end{equation}
which indicates
\begin{equation}\label{love3d}
\begin{array}{lll}
I\succ(H^{1/2}T_kH^{-1/2})^*(H^{-1/2}T_kH^{1/2})\ \ \ \
k=1,2,\cdots,m.
\end{array}
\end{equation}
As a result,
\begin{equation}\label{love3e}
\begin{array}{lll}
\|H^{1/2}T_kH^{-1/2}\|_2<1\ \ \ \ k=1,2,\cdots,m.
\end{array}
\end{equation}
Therefore, we have with (\ref{love3}) and (\ref{love3e}) that
\begin{equation}\label{love3h}
\begin{array}{lll}
\rho(T)&=&\rho(H^{1/2}TH^{-1/2})\leq\|H^{1/2}TH^{-1/2}\|_2\\&=&\|\sum_{k=1}^m\beta_kH^{1/2}T_kH^{-1/2}\|_2\\
&\leq&\sum_{k=1}^m\beta_k\|H^{1/2}T_kH^{-1/2}\|_2\\
&<&\sum_{k=1}^m\beta_k=1,
\end{array}
\end{equation}
which shows Algorithm \ref{algorithm1} is convergent. This completes
the proof.
\end{proof}

In Theorems \ref{theorem8} and \ref{theorem9}, some conditions such
that Algorithm \ref{algorithm1} converges have been presented. But,
it is difficult for us to construct a multisplitting such that these
conditions are easy to determine since they concern very complex
matrix operations. In the following, we will propose a practical
condition which is are easy to determine such that Algorithm
\ref{algorithm1} converges.

For the convenience of the proof of the following theorems, we will
introduce a type of block matrices-----{\em extended $H-$matrices}
which is further extension of {\em generalized $M-$matrices} and
{\em generalized $H-$matrices} introduced by Elsner and Mehrmann
\cite{L.V5} and Nabben \cite{R.N11}, respectively.

\begin{definition}  {\rm (see \cite{{L.V5}})}
\begin{enumerate}
\item $Z_m^k=\{A=[A_{ij}]\in \complex^{km\times
km}\ |\ A_{ij}\in C^{k\times k}\ is\ Hermitian\ for\ all\ i,j\in
N=\{1,2,\cdots,m\}\ and$ $A_{ij}\preceq0\ for\ all\ i\neq j,\ i,j\in
N\}$;
\item $\widehat{Z}_m^k=\{A=[A_{ij}]\in Z_m^k\ |\ A_{ii}\succ0,\
i\in N\}$;
\item $M_m^k=\{A\in \widehat{Z}_m^k\ |\ there\ exists\ u\in \mathbb{R}_+^m\ such\ that\ \sum_{j=1}^{m}u_jA_{ij}\succ0\ for\ all\
i\in N\}$, where $\mathbb{R}_+^m$ denotes all positive vectors in
$\mathbb{R}^m$, and a matrix $A\in \widehat{Z}_m^k$ is called a
generalized $M-$matrix if $A\in M_m^k$.
\end{enumerate}
\end{definition}

\begin{definition}
We define a set of $n\times n$ matrices $\Omega^n=\{\ A\in
\complex^{n\times n}\ |\ there\ exsits$ $a~nonsingular~matrix~C\in
\complex^{n\times n}~such~that~A=C^*DC,~where~
D=diag(d_1,\cdots,d_n)$ $\in\complex^{n\times n}\}$. Let $A\in
\Omega^n$. {\rm Then there must exsit a~nonsingular~matrix~$C\in
\complex^{n\times n}$~such~that  $A=C^*DC$,~where~
$D=diag(d_1,\cdots,d_n)\in \complex^{n\times n}$}. Define
\begin{equation}\label{love3i}
\begin{array}{lll}
\langle A\rangle:=C^*|D|C,
\end{array}
\end{equation}
where $|D|=diag(|d_1|,\cdots,|d_n|)\in \mathbb{R}^{n\times n}$.
\end{definition}

\begin{remark}
The set $\Omega^n$ includes many families of matrices such as
unitary matrices, Hermitian matrices, skew-Hermitian matrices,
normal matrices and positive definite matrices (not necessarily
Hermitian, see Theorem 3 in \cite{L.J}).
\end{remark}

\begin{definition}
\begin{enumerate}
\item $\Phi_m^k=\{A=[A_{ij}]\in \complex^{km\times
km}\ |\ A_{ij}\in \Omega^k\ for\ all\ i,j\in N=\{1,2,\cdots,m\}$
$and\ A_{ii}\ is\ positve\ definite\ for\ all\ i\in N\}$;
\item $H_m^k=\{A\in \Phi_m^k\ |\ \mu(A)\in M_m^k\}$, where $\mu(A)=[M_{ij}]\in \complex^{mk\times mk}$ is the block comparison matrix of $A$ and defined as
$$
M_{ij}:=\left\{
\begin{array}{cc}
(A_{ii}+A_{ii}^*)/2,\ \ & \ \ {\rm if}\ \ i=j\\
\langle{A}_{ij}\rangle, \ \ & \ \ {\rm if} \ \ i\neq j
\end{array} \right.,$$  and a
matrix $A\in \Phi_m^k$ is called an extended $H-$matrix if $A\in
H_m^k$.
\end{enumerate}
\end{definition}

\begin{lemma}\label{lemma4}
Let $B=C^*DC\in \complex^{n\times n}$ with $C\in \complex^{n\times
n}$ nonsingular and $D={\rm diag}(d_1,\ldots,d_n)\in
\complex^{n\times n}$, and let $\langle{B}\rangle=C^*|D|C\in
\complex^{n\times n}$ with $|D|={\rm diag}(|d_1|,\ldots,|d_n|)$.
Then the Hermitian matrix $\mathscr{B}=
\left \lbrack \begin{array}{cc} \langle{B}\rangle & e^{it}B \\
                                e^{-it}B^* & \langle{B}\rangle \end{array}\right \rbrack$
is positive semidefinite for all $t\in \mathbb{R}$.
\end{lemma}

\begin{proof}
Observe that $\mathscr{B}$ can be decomposed as
\begin{equation}\label{eq3w}
\begin{array}{llll}
\mathscr{B}&=&
\left \lbrack \begin{array}{cc} \langle{B}\rangle & e^{it}B \\
                                e^{-it}B^* & \langle{B}\rangle \end{array}\right \rbrack=\left \lbrack \begin{array}{cc} C^* & 0 \\
                                0 & C^* \end{array}\right
                                \rbrack\left \lbrack \begin{array}{cc} |D| & e^{it}D \\
                                e^{-it}D^* & |D| \end{array}\right
                                \rbrack\left \lbrack \begin{array}{cc} C & 0 \\
                                0 & C \end{array}\right
                                \rbrack\\&=&\mathscr{C}^*\left \lbrack \begin{array}{cc} |D| & e^{it}D \\
                                e^{-it}D^* & |D| \end{array}\right
                                \rbrack\mathscr{C},
                                \end{array}
\end{equation}
where $\mathscr{C}=\left \lbrack \begin{array}{cc} C & 0 \\
                                0 & C \end{array}\right
                                \rbrack$ is nonsingular since $C$ is.
Writing
$\mathscr{D} = \left \lbrack \begin{array}{cc} |D| & e^{it}D \\
                                e^{-it}D^* & |D| \end{array}\right
                                \rbrack$,
(\ref{eq3w}) shows that the Hermitian matrices $\mathscr{B}$ and
$\mathscr{D}$ are congruent, and therefore they must have the same
inertia. Hence, all we need to show is that $\mathscr{D}$ is
positive semidefinite.  Letting $\mathscr{P}$ denote the odd-even
permutation matrix of order $2n$, it is immediate to see that
$$\mathscr{P}^*\mathscr{D}\mathscr{P}
=\left \lbrack \begin{array}{cc} |d_1| & e^{it}d_1 \\ e^{-it}\bar
d_1 & |d_1|\end{array}\right
 \rbrack \oplus \cdots  \oplus
\left \lbrack \begin{array}{cc} |d_n| & e^{it}d_n \\ e^{-it}\bar d_n
& |d_n|\end{array}\right
 \rbrack  .$$
Hence, $\mathscr{P}^*\mathscr{D}\mathscr{P}$ is just a direct sum of
$n$ two-by-two Hermitian matrices, each of which is obviously
positive semidefinite. This shows that $\mathscr{D} \succeq 0$, and
the proof is complete.
\end{proof}

\begin{lemma}\label{lemma1}
Let $A=[A_{ij}]\in {\Phi}_m^k$. For $t\in R,$ define
\begin{equation}\label{1wd2}{A}_t=B+B^*-(e^{it}C+e^{-it}C^*),\end{equation}
 where $B=diag(B_{11},\cdots,B_{mm})$ and $C=B-{A}=[C_{ij}]\in \complex^{km\times
km}$ with $C_{ii}=C_{ii}^*$ for all $i=1,\cdots,m$. Let
$\tilde{C}=[\tilde{C}_{ij}]\in \complex^{km\times km}$ with
$\tilde{C}_{ij}=\langle C_{ij}\rangle$ for all $i,j=1,\cdots,m,$ and
let $\tilde{A}=B-\tilde{C}$. If $\mu(\tilde{A})+\mu(\tilde{A}^*)\in
M_m^k$, then $A_t\succ 0$ for all $t\in R.$
\end{lemma}

\begin{proof} Let $\tilde{A}=[\tilde{A}_{ij}]\in \complex^{km\times km}$. Then
\begin{equation}\label{aad1}
\tilde{A}_{ii}=B_{ii}-\tilde{C}_{ii}=B_{ii}-\langle{C}_{ii}\rangle\
\ {\rm and} \ \
\tilde{A}_{ij}=\tilde{C}_{ij}=\langle{C}_{ij}\rangle=\langle{A}_{ij}\rangle,\
\ i\neq j,
\end{equation}
for $i,j=1,\cdots,m.$ Since $\mu(\tilde{A})+\mu(\tilde{A}^*)\in
M_m^k$, there exists a vector $v=(v_1,\cdots,v_m)^T\in R^m_+$ such
that
\begin{equation}\label{61p}
v_i(\tilde{A}_{ii}+\tilde{A}_{ii}^*)-\sum\limits_{j=1,j\neq
i}^{m}v_j(\langle \tilde{A}_{ij}\rangle+\langle
A_{ji}^*\rangle)\succ0
\end{equation}
for all $i\in N=\{1,\cdots,m\}$. Multiply the inequality (\ref{61p})
by $v_i$ and define $V=diag(v_1I_k,\cdots,v_mI_k)$, where $I_k$ is
$k\times k$ identity matrix, such that $K=V\tilde{A}V$ satisfies
\begin{equation}\label{1eq3}
v_i^2(\tilde{A}_{ii}+\tilde{A}_{ii}^*)-\sum\limits_{j=1,j\neq
i}^{m}[\langle
v_i\tilde{A}_{ij}v_j\rangle+\langle(v_j\tilde{A}_{ji}v_i)^*\rangle]\succ0
\end{equation}
for all $i\in N$. Let $K=V\tilde{A}V=[K_{ij}]$ with
$K_{ij}=v_i\tilde{A}_{ij}v_j$ for all $i,j\in N$. Then following
(\ref{1eq3}), we have
\begin{equation}\label{2b111}
R_i(K)=(K_{ii}+K_{ii}^*)-\sum\limits_{j=1,j\neq i}^{m}[\langle
K_{ij}\rangle+\langle K_{ji}\rangle]\succ0,\ \ \ i=1,\cdots,m.
\end{equation}
Furthermore, according to (\ref{1wd2}), we have
\begin{equation}\label{multi35}
\begin{array}{llll}
K_t&=&VA_tV=V[B+B^*-(e^{it}C+e^{-it}C^*)]V\\
&=&V[B+B^*-(e^{it}+e^{-it})D_C-(e^{it}\hat{C}+e^{-it}\hat{C}^*)]V,
 \end{array}
 \end{equation}
 where $D_C=diag(C_{11},\cdots,C_{mm})$ and $\hat{C}=C-D_C.$
Since $C={A}-B$ and $C_{ii}=C_{ii}^*$ for all $i=1,\cdots,m$,
$C_{ii}=B_{ii}-A_{ii}$ and hence, coming from (\ref{aad1}), we have
\begin{equation}\label{multi36}
\begin{array}{llll}
B_{ii}+B^*_{ii}-(e^{it}+e^{-it})C_{ii}\succeq
B_{ii}+B^*_{ii}-2\langle
C_{ii}\rangle=\tilde{A}_{ii}+\tilde{A}^*_{ii}\succ 0
 \end{array}
 \end{equation}
for all $i=1,\cdots,m$. Thus,
\begin{equation}\label{multi37}
\begin{array}{llll}
B+B^*-(e^{it}+e^{-it})D_C\succeq D_{\tilde{A}}+D_{\tilde{A}}^*\succ
0,
 \end{array}
 \end{equation}
where $D_{\tilde{A}}=diag(\tilde{A}_{11},\cdots,\tilde{A}_{mm}).$
(\ref{aad1}), (\ref{multi35}) and (\ref{multi37}) imply
\begin{equation}\label{multi38}
\begin{array}{llll}
K_t&\succeq &V[D_A+D_A^*-(e^{it}\hat{C}+e^{-it}\hat{C}^*)]V\\
&=&\Delta+\sum\limits_{ i>j}R_{ij} +\sum\limits_{ i<j}S_{ij},
 \end{array}
 \end{equation}
where \begin{equation}\label{d122}
\begin{array}{llll}
\Delta&=&diag\{R_1(K),\cdots,R_m(K)\},\\
R_{ij}&=&\left[
 \begin{array}{ccccccccc}
 \ 0 & \cdots & 0 & 0 & \cdots & 0 & 0 & \cdots & 0 \\
 \ \vdots & \ddots & \vdots & \vdots & \vdots & \vdots & \vdots & \vdots & \vdots \\
 \ 0 & \cdots & \langle K_{ij}\rangle & 0& \cdots & e^{-it}K^*_{ij} & 0 & \cdots &0\\
 \ 0 & \cdots & 0 & 0 & \cdots & 0 & 0 & \cdots & 0 \\
\ \vdots & \vdots & \vdots & \vdots & \ddots & \vdots & \vdots & \vdots & \vdots \\
 \ 0 & \cdots & e^{it}K_{ij} & 0& \cdots & \langle K_{ij}\rangle & 0 & \cdots &0\\
 \ 0 & \cdots & 0 & 0 & \cdots & 0 & 0 & \cdots & 0 \\
\ \vdots & \vdots & \vdots & \vdots & \vdots & \vdots & \vdots & \ddots & \vdots \\
 \ 0 & \cdots & 0 & 0 & \cdots & 0 & 0 & \cdots & 0
 \end{array}
 \right]
  \end{array}
 \end{equation} and
 \begin{equation}\label{d123}
 S_{ij}=\left[
 \begin{array}{ccccccccc}
 \ 0 & \cdots & 0 & 0 & \cdots & 0 & 0 & \cdots & 0 \\
 \ \vdots & \ddots & \vdots & \vdots & \vdots & \vdots & \vdots & \vdots & \vdots \\
 \ 0 & \cdots & \langle K_{ij}\rangle & 0& \cdots & e^{it}K_{ij} & 0 & \cdots &0\\
 \ 0 & \cdots & 0 & 0 & \cdots & 0 & 0 & \cdots & 0 \\
\ \vdots & \vdots & \vdots & \vdots & \ddots & \vdots & \vdots & \vdots & \vdots \\
 \ 0 & \cdots & e^{-it}K^*_{ij} & 0& \cdots & \langle K_{ij}\rangle & 0 & \cdots &0\\
 \ 0 & \cdots & 0 & 0 & \cdots & 0 & 0 & \cdots & 0 \\
\ \vdots & \vdots & \vdots & \vdots & \vdots & \vdots & \vdots & \ddots & \vdots \\
 \ 0 & \cdots & 0 & 0 & \cdots & 0 & 0 & \cdots & 0
 \end{array}
 \right].
 \end{equation}
Since (\ref{2b111}) yields $\Delta\succ0$ and Lemma \ref{lemma4}
indicates $R_{ij}\succeq0$ and $ S_{ij}\succeq0$, $K_t\succ0$ and
hence $A_t\succ0$ for all $t\in R.$ This completes the proof.
\end{proof}

\begin{lemma}\label{lemma2} {\rm (see \cite{{R.N11}})}
Let $A,\ M,\ N\in C^{n\times n}$ with $A=M-N$. If for all $t\in R$
\begin{equation}\label{1k1}A_t:=M+M^H-(e^{it}N+e^{-it}N^H)>0,
\end{equation}
then $\rho(M^{-1}N)<1.$ If $A_t\geq 0$ for all $t\in R$, then
$\rho(M^{-1}N)\leq1.$
\end{lemma}

\begin{lemma}\label{lemma3}
Let $A=[A_{ij}]\in M_m^k$. Then there exist two positive diagonal
matrices $E=diag(E_1,\cdots,E_m)$ and $F=diag(F_1,\cdots,F_m)$,
where $E_i=e_iI\in \mathbb{R}^{k\times k}$ and $F_i=f_iI\in
\mathbb{R}^{k\times k}$ with $e_i>0$ and $f_i>0$ for all
$i=1,\cdots,m,$ such that $EAF+FA^*E\in M_m^k.$
\end{lemma}

\begin{proof}
The proof can be immediately obtained from Lemma 3.1 in
\cite{Huang}.
\end{proof}

\begin{theorem}\label{theorem1x}
Let $A=[A_{ij}]\in {\Phi}_m^k$, $B=diag(B_{11},\cdots,B_{mm})$ and
$C=B-{A}=[C_{ij}]\in \complex^{km\times km}$ with $C_{ii}=C_{ii}^*$
for all $i=1,\cdots,m$. Assume $\tilde{C}=[\tilde{C}_{ij}]\in
\complex^{km\times km}$ with $\tilde{C}_{ij}=\langle C_{ij}\rangle$
for all $i,j=1,\cdots,m,$ and let $\tilde{A}=B-\tilde{C}$. If
$\tilde{A}\in H_m^k$, then $\rho(B^{-1}C)<1.$
\end{theorem}

\begin{proof} Since $\tilde{A}\in H_m^k$, $\mu(\tilde{A})\in M_m^k$. Lemma
\ref{lemma3} shows that there exist two positive diagonal matrices
$E=diag(E_1,\cdots,E_m)$ and $F=diag(F_1,\cdots,F_m)$, where
$E_i=e_iI\in \mathbb{R}^{k\times k}$ and $F_i=f_iI\in
\mathbb{R}^{k\times k}$ with $e_i>0$ and $f_i>0$ for all
$i=1,\cdots,m,$ such that
$E\mu(\tilde{A})F+F\mu(\tilde{A})^*E=\mu(E\tilde{A}F)+\mu(F\tilde{A}^*E)\in
M_m^k.$ $A=B-C$ yields $EAF=EBF-ECF$. Let $\hat{{A}}=EAF,\
\hat{{B}}=EBF$ and $\hat{{C}}=ECF$. Since
$\hat{{B}}=EBF=diag(\hat{{B}}_{11},\cdots,\hat{{B}}_{mm})$ with
$\hat{{B}}_{ii}=e_i{B}_{ii}f_i$ for all $i=1,\cdots,m$ and
$\hat{{C}}=\hat{{B}}-\hat{{A}}$, it follows from Lemma \ref{lemma1}
that
\begin{equation}\label{multi23}{\hat{{A}}}_t=\hat{{B}}+\hat{{B}}^*-(e^{it}\hat{{C}}+e^{-it}\hat{{C}}^*)\succ0\end{equation}
 for all $t\in R.$ Lemma \ref{lemma2} shows that
 $\rho(\hat{B}^{-1}\hat{C})<1.$ Again, $$\hat{B}^{-1}\hat{C}=(EBF)^{-1}(ECF)=F^{-1}(B^{-1}C)F.$$
Then $\hat{B}^{-1}\hat{C}$ and $B^{-1}C$ have the same eigenvalues.
As a result,
$$\rho({B}^{-1}{C})=\rho(\hat{B}^{-1}\hat{C})<1$$ which shows that we complete
the proof.
\end{proof}

\begin{lemma}\label{lemma5}
Let $A=(a_{ij})\in \complex^{n\times n}$ with a multisplitting
$(M_k,N_k,E_k)_{k=1}^{m}$, and let $T=\sum_{k=1}^{m}E_kM_k^{-1}N_k$
and $\hat{A}=B-C$, where
\begin{equation}\label{r12}
\begin{array}{llll}
B= \left[
 \begin{array}{cccc}
 \ M_1 & 0 & \cdots & 0  \\
 \ 0 & M_2 & \cdots & 0  \\
 \ \vdots & \vdots & \ddots & \vdots  \\
 \ 0 & 0 & \cdots & M_m  \\
 \end{array}
 \right],\
 C= \left[
 \begin{array}{cccc}
 \ N_1E_1 & N_1E_2 & \cdots & N_1E_m  \\
 \ N_2 E_1 & N_2E_2 & \cdots & N_2E_m  \\
 \ \vdots & \vdots & \ddots & \vdots  \\
 \ N_m E_1 & N_mE_2 & \cdots & N_mE_m  \\
 \end{array}
 \right].
\end{array}
\end{equation}
Then $\rho(T)=\rho(B^{-1}C)$, where $\rho(T)$ denotes the spectral
radius of the matrix $T$.
\end{lemma}

\begin{proof} \begin{equation}\label{r13}
\begin{array}{llll}
\rho(T)&=& \rho(\sum\limits_{k=1}^{m}E_kM_k^{-1}N_k)\\
&=&\rho\left(\left[
 \begin{array}{cccc}
 \ E_1 & E_2 & \cdots & E_m  \\
 \ 0 & 0 & \cdots & 0  \\
 \ \vdots & \vdots & \ddots & \vdots  \\
 \ 0 & 0 & \cdots & 0  \\
 \end{array}
 \right]\left[
 \begin{array}{cccc}
 \ M_1^{-1}N_1 & 0 & \cdots & 0  \\
 \ M_2^{-1}N_2 & 0 & \cdots & 0  \\
 \ \vdots & \vdots & \ddots & \vdots  \\
 \ M_m^{-1}N_m & 0 & \cdots & 0  \\
 \end{array}
 \right]\right)\\
&=&\rho\left(\left[
 \begin{array}{cccc}
 \ M_1^{-1}N_1 & 0 & \cdots & 0  \\
 \ M_2^{-1}N_2 & 0 & \cdots & 0  \\
 \ \vdots & \vdots & \ddots & \vdots  \\
 \ M_m^{-1}N_m & 0 & \cdots & 0  \\
 \end{array}
 \right]\left[
 \begin{array}{cccc}
 \ E_1 & E_2 & \cdots & E_m  \\
 \ 0 & 0 & \cdots & 0  \\
 \ \vdots & \vdots & \ddots & \vdots  \\
 \ 0 & 0 & \cdots & 0  \\
 \end{array}
 \right]\right)\end{array}
\end{equation}
\begin{eqnarray*}
 &=&\rho\left(\left[
 \begin{array}{cccc}
 \ M_1^{-1}N_1E_1 & M_1^{-1}N_1E_2 & \cdots & M_1^{-1}N_1E_m  \\
 \ M_2^{-1}N_2 E_1 & M_2^{-1}N_2E_2 & \cdots & M_2^{-1}N_2E_m  \\
 \ \vdots & \vdots & \ddots & \vdots  \\
 \ M_m^{-1}N_m E_1 & M_m^{-1}N_mE_2 & \cdots & M_m^{-1}N_mE_m  \\
 \end{array}
 \right]\right)\\
 &=&\rho(B^{-1}C),
\end{eqnarray*}
where $B$ and $C$ are defined as (\ref{r12}). This completes the
proof.
\end{proof}

\begin{lemma}\label{lemma6} {\rm (see Corollary 7.6.5 in \cite{{R.A4}})}
Let $A, B\in \complex^{n\times n}$ be Hermitian and $A\succ0$. Then
there exists a nonsingular matrix $C\in \complex^{n\times n}$ such
that $A=C^*C$ and $B=C^*DC$, where $D\in \mathbb{R}^{n\times n}$ is
a diagonal matrix.
\end{lemma}

\begin{theorem}\label{theorem2}
Let $A\in \complex^{n\times n}$ be non-Hermitian positive definite
with a multisplitting $(M_k,N_k,E_k)_{k=1}^{m}$ satisfying the
condition that $E_k=\beta_k I$ and $A=M_k-N_k$ is a P-regular
splitting with $N_k$ Hermitian for all $k=1,2,\cdots,m.$ Then
Algorithm \ref{algorithm1} converges to the unique solution of
{\rm(\ref{r1})} for any choice of the initial guess $x^{(0)}$.
\end{theorem}

\begin{proof}
We only prove $\rho(T)<1$. Lemma 3.1 shows $\rho(T)=\rho(B^{-1}C)$,
where $B$ and $C$ are defined as (\ref{r12}), and $\hat{A}=B-C$.
Since $E_k=\beta_k I$, \begin{equation}\label{multi24}
\begin{array}{llll}
\hat{A}&=&B-C\\ &=&\left[
 \begin{array}{cccc}
 \ M_1 & 0 & \cdots & 0  \\
 \ 0 & M_2 & \cdots & 0  \\
 \ \vdots & \vdots & \ddots & \vdots  \\
 \ 0 & 0 & \cdots & M_m  \\
 \end{array}
 \right]-\left[
 \begin{array}{cccc}
 \ \beta_1N_1 & \beta_2N_1 & \cdots & \beta_mN_1  \\
 \ \beta_1N_2  & \beta_2N_2 & \cdots & \beta_mN_2  \\
 \ \vdots & \vdots & \ddots & \vdots  \\
 \ \beta_1N_m  & \beta_2N_m & \cdots & \beta_mN_m  \\
 \end{array}
 \right].
\end{array}
\end{equation}
Let $H(M_k)=(M_k+M^*_k)/2$ be the Hermitian part of $M_k$ for
$k=1,2,\cdots,m$. Since $A$ is non-Hermitian positive definite and
$A=M_k-N_k$ is $P$-regular splittings with $N_k$ Hermitian for all
$k=1,2,\cdots,m$, one has
\begin{equation}\label{qing4}
\begin{array}{lll}
&&H(M_k)+N_k\succ 0,\ \ H(M_k)-N_k\succ 0.
\end{array}
\end{equation}
(\ref{qing4}) shows that $H(M_k)\succ0$. Also, $N_k$ is Hermitian.
It follows from Lemma \ref{lemma6} that there exists a nonsingular
matrix $C_k\in \complex^{n\times n}$ such that $H(M_k)=C_k^*C_k$ and
$N_k=C_k^*D_kC_k$, where $D_k\in \mathbb{R}$ are diagonal matrices.
Following (\ref{qing4}), we have
\begin{equation}\label{qing3}
\begin{array}{lll}
C_k^*(I+D_k)C_k\succ 0,\ \ C_k^*(I-D_k)C_k\succ 0
\end{array}
\end{equation}
and consequently,\begin{equation}\label{qing3a}
\begin{array}{lll}
I+D_k\succ 0,\ \ I-D_k\succ 0,
\end{array}
\end{equation}
which shows that \begin{equation}\label{qing3b}
\begin{array}{lll}
I-|D_k|\succ 0.
\end{array}
\end{equation}
As a consequence,
\begin{equation}\label{qing3c}
\begin{array}{lll}
H(M_k)-\langle{N_k}\rangle=C_k^*(I-|D_k|)C_k\succ0.
\end{array}
\end{equation}
This leads to \begin{equation}\label{qing3g}
\begin{array}{lll}
H(M_k)-\langle{N_k}\rangle=(H(M_k)-\beta_k\langle{N_k}\rangle)-\sum_{j=1,j\neq
k}^m(\beta_j\langle{N_k}\rangle)\succ0
\end{array}
\end{equation}
for all $k=1,2,\cdots,m.$ Let $\tilde{C}=[\tilde{C}_{ij}]\in
\complex^{km\times km}$ with $\tilde{C}_{ij}=\langle
\beta_jN_i\rangle$ for all $i,j=1,\cdots,m,$ and let
$\tilde{A}=B-\tilde{C}$. (\ref{qing3g}) shows $\tilde{A}\in H_m^k$.
It follows from Theorem \ref{theorem1x} that
$\rho(T)=\rho(B^{-1}C)<1$. Therefore, Algorithm \ref{algorithm1} is
convergent. This completes the proof.
\end{proof}

Following Theorem \ref{theorem2}, a corollary is obtained
immediately.

\begin{corollary}\label{corollary2a}
Let $A\in \complex^{n\times n}$ be non-Hermitian positive definite
with a multisplitting $(M_k,N_k,E_k)_{k=1}^{m}$ satisfying the
condition that $E_k=\beta_k I$ and $N_k\succeq0$ for all
$k=1,2,\cdots,m.$ Then Algorithm \ref{algorithm1} converges to the
unique solution of {\rm(\ref{r1})} for any choice of the initial
guess $x^{(0)}$.
\end{corollary}

Now, we study the convergence of the relaxed multisplitting method.

\begin{theorem} \label{theorem8a}
Let $A\in \complex^{n\times n}$ be nonsingular with a multisplitting
$(M_k,N_k,E_k)_{k=1}^{m}$ satisfying one of the following
conditions:

{\rm(i)} $E_k=\beta_k I$ and $M_k^*A+A^*N_k=M_k^*M_k-N_k^*N_k\succ0$
for $k=1,2,\cdots,m$;

{\rm(ii)} $E_k=\beta_k I$ and $\|M_k^{-1}N_k\|_2<1$ for
$k=1,2,\cdots,m$.\\
Then the relaxed multisplitting algorithm converges to the unique
solution of {\rm(\ref{r1})} for any choice of the initial guess
$x^{(0)}$, provided $\omega\in (0,2/(1+\rho))$, where $\rho=\rho(T)$
and $T$ is defined in (\ref{multi3}).
\end{theorem}

\begin{proof}
Since the iteration matrix of the relaxed multisplitting algorithm
is $T_{\omega}=\omega T+(1-\omega)I$, where $T$ is the iteration
matrix of Algorithm \ref{algorithm1}. Let $\lambda_i,\
i=1,\cdots,n,$ be eigenvalues of $T$, then
$\omega\lambda_i+(1-\omega),\ i=1,\cdots,n,$ are eigenvalues of
$T_{\omega}$. Assume
$\rho(T_{\omega})=|\omega\lambda_1+(1-\omega)|$. Then
$\rho(T_{\omega})\leq\omega|\lambda_1|+|1-\omega|\leq\omega\rho(T)+|1-\omega|.$
Since Theorem \ref{theorem8} gives $\rho(T)<1$ and $\omega\in
(0,2/(1+\rho(T)))$, $\rho(T_\omega)<1$, which shows that the relaxed
multisplitting algorithm converges for any initial vector $x^{(0)}$,
provided $\omega\in (0,2/(1+\rho))$.
\end{proof}

\begin{theorem} \label{theorem9a}
Let $A\in \complex^{n\times n}$ be non-Hermitian positive definite
with a multisplitting $(M_k,N_k,E_k)_{k=1}^{m}$ satisfying one of
the following conditions:

{\rm(i)} $E_k=\beta_k I$ and $A=M_k-N_k$ is an extended P-regular
splitting for all $k=1,2,\cdots,m$;

{\rm(ii)} $E_k=\beta_k I$ and $A=M_k-N_k$ is a P-regular splitting
with $N_k$ Hermitian for all $k=1,2,\cdots,m;$

{\rm(iii)} $E_k=\beta_k I$ and $N_k\succeq 0$ for all
$k=1,2,\cdots,m$.\\
Then the relaxed multisplitting algorithm converges to the unique
solution of {\rm(\ref{r1})} for any choice of the initial guess
$x^{(0)}$, provided $\omega\in (0,2/(1+\rho))$, where $\rho=\rho(T)$
and $T$ is defined in (\ref{multi3}).
\end{theorem}

\begin{proof}
Similar to the proof of Theorem \ref{theorem8a}, the proof can be
obtained immediately from Theorem \ref{loveluo1}, Theorem
\ref{theorem2} and Corollary \ref{corollary2a}.
\end{proof}

\section{Convergence of parallel PSS method}\label{convergence1-sec}
In this section some convergence results on the parallel PSS methods
for non-Hermitian positive definite linear systems will be
presented. The following lemma will be used in this section.

\begin{lemma}\label{lemma11a} {\rm (see \cite{{B.Z3}})}
Let $A\in \complex^{n\times n}$ be non-Hermitian positive definite,
let $M(\alpha_k)$ defined in (\ref{multi8}) be the iteration matrix
of the PSS iteration, and let
\begin{equation}\label{mult1}
\begin{array}{lll}
V(\alpha_k)=(\alpha_kI-P_k)(\alpha_kI+P_k)^{-1}.
\end{array}
\end{equation}
Then the spectral radius $\rho(M(\alpha_k))$ of $M(\alpha_k)$ is
bound by $\|V(\alpha_k)\|_2$. Therefore, it holds that
\begin{equation}\label{mult2}
\begin{array}{lll}
\rho(M(\alpha_k))\leq \|M(\alpha_k)\|_2\leq\|V(\alpha_k)\|_2<1,\ \ \
\forall \alpha_k>0.
\end{array}
\end{equation}
\end{lemma}

\begin{theorem}\label{theorem20}
Let $A\in \complex^{n\times n}$ be non-Hermitian positive definite
with a multisplitting $(M_k,N_k,E_k)_{k=1}^{m}$ such that
(\ref{multi05}) holds and $E_k=\beta_k I$ and $\alpha_k\geq0$ for
all $k=1,2,\cdots,m.$ Then Algorithm \ref{algorithm2} converges to
the unique solution of {\rm(\ref{r1})} for any choice of the initial
guess $x^{(0)}$.
\end{theorem}

\begin{proof}
Observing that the iteration matrix of Algorithm \ref{algorithm2} is
defined as $\mathscr{M}(\alpha)=\sum_{k=1}^mE_kM(\alpha_k)$ and
$E_k=\beta_k I$ and $\alpha_k\geq0$ for all $k=1,2,\cdots,m.$ Then
Lemma \ref{lemma11a} shows that $\rho(M(\alpha_k))\leq
\|M(\alpha_k)\|_2<1$. As a consequence,
\begin{equation}\label{lovea22}
\begin{array}{lll}
\rho(\mathscr{M}(\alpha))&\leq &\|\mathscr{M}(\alpha)\|_2=\|\sum_{k=1}^mE_kM(\alpha_k)\|_2\\
&\leq&\sum_{k,=1}^m\|E_kM(\alpha_k)\|_2\\
&=&\sum_{k=1}^m\beta_k\|M(\alpha_k)\|_2\\
&<&\sum_{k=1}^m\beta_k=1,
\end{array}
\end{equation}
which shows that Algorithm \ref{algorithm2} is convergent. This
completes the proof.
\end{proof}

\begin{theorem}\label{theorem2h}
Let $A\in \complex^{n\times n}$ be non-Hermitian positive definite
with a multisplitting $(M_k,N_k,E_k)_{k=1}^{m}$ such that
(\ref{multi05}) holds, $E_k=\beta_k I$ and $\alpha_k\geq0$ for all
$k=1,2,\cdots,m.$ Then the Relaxed Parallel PSS Algorithm converges
to the unique solution of {\rm(\ref{r1})} for any choice of the
initial guess $x^{(0)}$, provided $\omega\in (0,2/(1+\rho))$, where
$\rho=\rho(\mathscr{M}(\alpha))$ and $\mathscr{M}(\alpha)$ is
defined in (\ref{multi10}).
\end{theorem}

\begin{proof}
Similar to the proof of Theorem \ref{theorem8a}, the proof can be
obtained immediately from Theorem \ref{theorem20}.
\end{proof}

\begin{remark}
As two special cases of the Parallel PSS method, the Parallel TSS
method and Parallel BTSS method together with their relaxed versions
are convergent.
\end{remark}

As is pointed out in \cite{B.Z3}, there are two important problems
to be further studied for the Parallel PSS method. One is the choice
of the skew-Hermitian matrix $S_k=N_k-N_k^*$, here is the choice of
the matrix $N_k$ such that $P_k=M_k+N_k^*$ is easily inverted. Here,
$N_k$ can be chosen as triangular or block triangular matrix (see
\cite{{B.Z3}}) such that $P_k=M_k+N_k^*$ is triangular or block
triangular matrix.

The other is the choice of the acceleration parameter $\alpha_k$
such that the Parallel PSS method converges very fast. If $P_k\in
\complex^{n\times n}$ is a normal matrix, then we can compute
$\widehat{\alpha}_k^{\star}=arg\
min_{\alpha>0}\{\|V(\alpha_k)\|_2\}$ by making use of the formula in
Theorem 2.2 of \cite{B.Z8}. But there is not such a formula as in
Theorem 2.2 of \cite{B.Z8} to compute a usable
$\widehat{\alpha}_k^{\star}$ if $P_k\in \complex^{n\times n}$ is a
general positive definite matrix, and hence, the upper bound
$\|V(\widehat{\alpha}_k^{\star})\|_2$. Now, we give such a formula
to compute $\widehat{\alpha}_k^{\star}$ and hence, the upper bound
$\|V(\widehat{\alpha}_k^{\star})\|_2$.

\begin{theorem}\label{theorem21}
Let $P_k\in \complex^{n\times n}$ be non-Hermitian positive definite
with $H_k=(P_k+P_k^*)/2$ its Hermitian part, and let $V_k(\alpha_k)$
be defined in (\ref{mult1}). Then it holds that
\begin{equation}\label{mult5}
\begin{array}{lll}
\widehat{\alpha}_k^{\star}=arg\
min_{\alpha_k>0}\{\|V(\alpha_k)\|_2\}=\sqrt{x^*P_k^*P_kx}\in
[\sigma^k_n,\sigma^k_1]
\end{array}
\end{equation}
and
\begin{equation}\label{mult6}
\begin{array}{lll}
\|V(\widehat{\alpha}_k^{\star})\|_2=min_{\alpha_k>0}\{\|V(\alpha_k)\|_2\}=\sqrt{\displaystyle\frac{\widehat{\alpha}_k^{\star}-x^*H_kx}{\widehat{\alpha}_k^{\star}+x^*H_kx}},
\end{array}
\end{equation}
where $x$ satisfies $\|x\|_2=1$ and $G^{-1}Kx=\rho(G^{-1}K)x$ with
$G=({\alpha}_k I+P_k)^*({\alpha}_k I+P_k)$ and $K=({\alpha}_k
I-P_k)^*({\alpha}_k I-P_k)$, and $\sigma^k_n$ and $\sigma^k_1$ are
the minimal and maximal singular values of $P_k$, respectively.
\end{theorem}

\begin{proof}
Since $\|V(\alpha_k)\|_2=\rho(V(\alpha_k)^*V(\alpha_k))$ and
(\ref{mult1}),
\begin{equation}\label{mult7}
\begin{array}{lll}
\|V(\alpha_k)\|_2^2&=&\rho((\alpha_kI+P^*_k)^{-1}(\alpha_kI-P_k)^*(\alpha_kI-P_k)(\alpha_kI+P_k)^{-1})\\
&=&\rho\{[(\alpha_kI+P_k)^*(\alpha_kI+P_k)]^{-1}(\alpha_kI-P_k)^*(\alpha_kI-P_k)\}\\
&=&\rho(G^{-1}K).
\end{array}
\end{equation}
$\|x\|_2=1$, $G^{-1}Kx=\rho(G^{-1}K)x$ and (\ref{mult7}) show that
it holds that
\begin{equation}\label{mult8}
\begin{array}{lll}
\|V(\alpha_k)\|^2_2
&=&\rho(G^{-1}K)=\displaystyle\frac{x^*Kx}{x^*Gx}\\
&=&\displaystyle\frac{x^*[({\alpha}_k
I-P_k)^*({\alpha}_k I-P_k)]x}{x^*[({\alpha}_k I+P_k)^*({\alpha}_k I+P_k)]x}\\
&=&\displaystyle\frac{{\alpha}^2_k-2{\alpha}_k
x^*H_kx+x^*P_k^*P_kx}{{\alpha}^2_k+2{\alpha}_k
x^*H_kx+x^*P_k^*P_kx}\\
&=&1-\displaystyle\frac{4{\alpha}_k
x^*H_kx}{{\alpha}^2_k+2{\alpha}_k
x^*H_kx+x^*P_k^*P_kx}\\
&=&1-f(\alpha_k),
\end{array}
\end{equation}
where $f(\alpha_k)=\displaystyle\frac{4{\alpha}_k
x^*H_kx}{{\alpha}^2_k+2{\alpha}_k x^*H_kx+x^*P_k^*P_kx}$. As a
result,
\begin{equation}\label{mult9}
\begin{array}{lll}
min_{\alpha_k>0}\{\|V(\alpha_k)\|_2\}=1-max_{\alpha_k>0}f(\alpha_k).
\end{array}
\end{equation}
Since
\begin{equation}\label{mult10}
\begin{array}{lll}
f'(\alpha_k)=\displaystyle\frac{4
x^*H_kx(x^*P_k^*P_kx-{\alpha}_k^2)}{({\alpha}^2_k+2{\alpha}_k
x^*H_kx+x^*P_k^*P_kx)^2},
\end{array}
\end{equation}
$f(\alpha_k)$ is gradually increasing if $\alpha_k\in (0,
\sqrt{x^*P_k^*P_kx})$, $f(\alpha_k)$ is gradually decreasing if
$\alpha_k\in (\sqrt{x^*P_k^*P_kx}, \infty)$ and consequently, when
$\alpha_k= \sqrt{x^*P_k^*P_kx}$, $f(\alpha_k)$ gets its maximum
$max_{\alpha_k>0}f(\alpha_k)=\displaystyle\frac{2x^*H_kx}{
x^*H_kx+\sqrt{x^*P_k^*P_kx}}$. Therefore, when
$$\widehat{\alpha}_k^{\star}=arg\
min_{\alpha_k>0}\{\|V(\alpha_k)\|_2\}=\sqrt{x^*P_k^*P_kx}\in
[\sigma^k_n,\sigma^k_1],$$
$$\|V(\widehat{\alpha}_k^{\star})\|_2=min_{\alpha_k>0}\{\|V(\alpha_k)\|_2\}=\sqrt{\displaystyle\frac{\widehat{\alpha}_k^{\star}-x^*H_kx}{\widehat{\alpha}_k^{\star}+x^*H_kx}},
$$ which shows that we complete the proof.
\end{proof}

\begin{remark}
Usually, it holds that
\begin{equation}\label{mult11}
\widehat{\alpha}_k^{\star}\neq\alpha_{opt}=arg\
min_{\alpha_k>0}\{\rho(M(\alpha_k))\}
\end{equation}
and \begin{equation}\label{mult12}
\rho(M(\widehat{\alpha}^{\star}_k))\geq\rho(M(\alpha_{opt})).
\end{equation}
\end{remark}


\section{Conclusions} \label{conclusions-sec}

In this paper we have studied the convergence of the parallel
mulitisplitting iterative methods and the parallel PSS methods for
the solution of non-Hermitian positive definite linear systems. Some
of our results can be regarded as generalizations of analogous
results for the Hermitian positive definite case.


{\bf Acknowledgments.} The first author would like to acknowledge
the hospitality of the Department of Mathematics and Computer
Science at Emory University, where this work was completed. Many
thanks also to Professor Michele Benzi for suggesting the topic of
this paper and for helpful suggestions.


\end{document}